\documentclass[11pt,a4paper]{amsart}

\usepackage{times} 
\usepackage{amsmath} 
\usepackage{amssymb}  
\usepackage{amsthm}
\usepackage{latexsym}
\usepackage{amsfonts}
\usepackage{xcolor}
\usepackage{graphicx}

\newtheorem{thm}{Theorem}
\newtheorem{cor}[thm]{Corollary}
\newtheorem{lem}[thm]{Lemma}
\newtheorem{prop}[thm]{Proposition}
\theoremstyle{definition}
\newtheorem{defn}[thm]{Definition}
\newtheorem{rem}[thm]{Remark}

\numberwithin{equation}{section}
\allowdisplaybreaks

\newcommand{\R}{\ensuremath{\mathbb R}}    
\newcommand{\C}{\ensuremath{\mathbb C}}    

\newcommand{\<}{\langle}
\renewcommand{\>}{\rangle}

\newcommand{\calH}{\mathcal H}

\newcommand{\calR}{\mathcal R}         
         
\newcommand{\calT}{\mathcal T}         
\newcommand{\calU}{\mathcal U}         
\newcommand{\calV}{\mathcal V}

\newcommand{\la}{\lambda}

\newcommand{\vphi}{\varphi}

\renewcommand{\Re}{\operatorname{Re}}

\newcommand{\dom}{\operatorname{dom}}

\newcommand{\wto}{\rightharpoonup}

\newcommand{\wt}{\widetilde}

\newcommand{\UU}{\mathbb U}

\newcommand{\inte}{\operatorname{int}}

\newcommand{\dist}{\operatorname{dist}}
\newcommand{\sd}{\sigma_{\rm d}}

\begin{document}
\title{Minimizing the energy supply of infinite-dimensional linear port-Hamiltonian systems}
\author[F.\ Philipp, M.\ Schaller, T.\ Faulwasser, B.\ Maschke and K.\ Worthmann]{Friedrich Philipp$^{1}$, Manuel Schaller$^{1}$, Timm Faulwasser$^{2}$, Bernhard Maschke$^{3}$, and Karl Worthmann$^{1}$}
	\thanks{}
	\thanks{$^{1}$Technische Universit\"at Ilmemau, Institute for Mathematics, Germany
		{\tt\small \{friedrich.philipp, manuel.schaller, karl.worthmann\}@tu-ilmenau.de}.} %
	\thanks{$^{2}$TU Dortmund University, Institute of Energy Systems, Energy Efficiency and Energy Economics, Germany
		{\tt\small timm.faulwasser@ieee.org}}%
	\thanks{$^{3}$Univ Lyon, Universit{\'e} Claude Bernard Lyon 1, CNRS, LAGEPP UMR 
		5007, France {\tt\small bernhard.maschke@univ-lyon1.fr}}
	\thanks{{\bf Acknowledgments: }F.\ Philipp was funded by the Carl Zeiss Foundation within the project \textit{DeepTurb---Deep Learning in und von Turbulenz}. 
		M.\ Schaller was funded by the DFG (grant WO\ 2056/2-1, project number 289034702, grant WO\ 2056/7-1, project number 430154635). %
		K.\ Worthmann gratefully acknowledges funding by the German Research Foundation (DFG; grant WO\ 2056/6-1, project number 406141926).
}

\begin{abstract}
We consider the problem of minimizing the supplied energy of infinite-dimensional linear port-Hamiltonian systems and prove that optimal trajectories exhibit the turnpike phenomenon towards certain subspaces induced by the dissipation of the dynamics.  

\end{abstract}

\maketitle

\section{Introduction}
The analysis of dynamics is catalysed by intrinsic properties such as passivity, dissipativity or port-Hamiltonian structures. 
It is worth to be noted that the system-theoretic notion of dissipativity as coined by~\cite{Willems71a,Willems72a}, in disguise, leverages optimal control, i.e., the available storage and the required supply are value functions of appropriately defined infinite-horizon optimal control problems (OCPs).

A recent line of research exploits the close relation between dissipativity and optimal control to analyse receding-horizon schemes~\cite{kit:faulwasser18c} and OCPs parametric in the initial conditions and/or horizon length. The latter leads to the recently re-popularized notion of turnpike properties of OCPs. In their easiest variant they refer to the phenomenon that optimal solutions of an OCP for different initial conditions or increasing horizons spend an increasing amount of time close to the optimal steady state, see~\cite{tudo:faulwasser21b} for a recent overview and historical remarks. One can also show that this steady state (a.k.a. the turnpike) corresponds to an equilibrium of the optimality system and that it is an asymptotically stable attractor of the infinite-horizon optimal solutions~\cite{Carlson91,tudo:faulwasser21a}.
Moreover, it turns out that an OCP specific (strict) dissipativity notion, in which the normalized stage cost is considered as supply rate, is key in analysing turnpike properties of OCPs~\cite{Stieler14a,epfl:faulwasser15h,Gruene16a}.
For turnpike properties of infinite dimensional OCPs, as considered below, the linear quadratic case is covered in numerous works, e.g.,  \cite{Breiten2018,Gugat2018,Gugat2016,Gruene2018c,Gruene2019} and for the nonlinear case see \cite{Esteve2020,Gruene2021,Pighin2020,Porretta2013,Trelat2016,Trelat18}. All of these references have in common that they consider non-singular problems and optimal equilibria as turnpike reference points.

At first glance, the research on OCPs analysis via dissipativity and the manifold results on port-Hamiltonian (pH) systems are disjoint. Yet, in a recent paper, we analysed OCPs for port-Hamiltonian systems considering the intrinsic pH objective of minimal energy supply. Intuitively, the (not-necessarily strict) dissipativity of such OCPs is clear, as pH-systems are passive with respect to the supplied energy.  However, in~\cite{Schaller2021a} we have shown that in case of linear finite-dimensional pH systems, the minimization of energy supply induces a singular OCP which is strictly dissipative with respect to a specific subspace. More precisely, the analysis in~\cite{Schaller2021a} leverages the port-Hamiltonian structure to verify strict dissipativity  with respect to the kernel of the dissipation matrix of the pH system.

%
%

%

In this work, we extend the analysis towards OCPs for  port-Hamil\-toni\-an systems on an infinite-dimensional Hilbert space $\calH$. 
The considered OCP reads as follows:
\begin{subequations}\label{e:ocp}
\begin{align}
\min_{u\in\calU_T}&\int_0^T\Re\<u(t),y(t)\>\,dt\label{e:cost}\\
\dot x(t) &= (J-R)x(t) + Bu(t)\label{e:dyn}\\
y(t) &= B^*x(t)\label{e:out}\\
x(0) &= x_0,\quad x(T) = x_T.\label{e:initial}
\end{align}
\end{subequations}
where $x_0,x_T\in\calH$, $R$ is a (possibly unbounded) non-negative self-adjoint operator and $J$ is a (possibly unbounded) skew-adjoint operator on $\calH$. The setting will be made precise in Section~\ref{s:ocp}. Under suitable assumptions, the operator $A := J-R$ is then a generator of a $C_0$-semigroup of contractions. Also, the control constraint $\calU_T$ will be introduced in Section~\ref{s:ocp}.
As the above problem is singular, i.e., there is no quadratic term of the control in the objective functional, established methods to derive turnpike properties, i.e., strict dissipativity with respect to a steady state or Riccati theory cannot be applied.
Hence the present paper transfers the subspace approach \cite{Schaller2021a} to the infinite-dimensional setting of~\eqref{e:ocp}. 

The remainder is structured as follows: Section \ref{s:ocp} details the infinite-dimensional setting of OCP \eqref{e:ocp} and gives sufficient conditions for the existence of an optimal solution. 
Section \ref{s:TP} first presents two numerical examples of subspace state turnpikes in instances of OCP \eqref{e:ocp}, before we turn towards  the proof of  sufficient conditions. 
The paper ends with a summary in Section \ref{sec:con}.

\section{The Optimal Control Problem}\label{s:ocp}
As already indicated in the Introduction, our main objective is to analyze port-Hamiltonian optimal control problems (OCPs) of the form \eqref{e:ocp}. To this end, let $\calH$ be a separable Hilbert space with scalar product $\langle\cdot,\cdot\rangle$ and corresponding norm $\|\cdot\|$. By $L(\calH)$ we denote the set of bounded linear operators from $\calH$ to itself. The {\em resolvent set} $\rho(T)$ of a linear operator $T$ in $\calH$ is defined by $\rho(T) := \{s\in\C : (T-s)^{-1}\in L(\calH)\}$. The {\em spectrum} of $T$ is the complement $\sigma(T) := \C\setminus\rho(T)$.

Let $J$ be a (possibly unbounded) skew-adjoint operator on $\calH$ (i.e., $J^* = -J$) and let $R$ be a positive semi-definite (possibly unbounded) self-adjoint operator on $\calH$ (i.e., $R^* = R\ge 0$). Clearly, both $J$ and $R$ are generators of contractive $C_0$-semigroups. Here, we assume that one of the following conditions is satisfied:
\begin{itemize}
	\item $J$ is $R$-bounded with $R$-bound smaller than one or
	\item $R$ is $J$-bounded with $J$-bound smaller than one.
\end{itemize}
In both cases, the operator
$$
A := J - R
$$
(defined on $\dom J\cap\dom R$) is a generator of a $C_0$-semigroup of contractions (see \cite[Thm.\ III.2.7]{engelnagel}) and therefore maximal dissipative. We shall denote this semigroup by $(\calT(t))_{t\ge 0}$.
%
%

Here, we consider the case of a finite number of input ports being linearly mapped to the infinite-dimensional state space, i.e., the operator
$$
B : \C^m\to\calH
$$
in \eqref{e:dyn} is linear and bounded.

Given $x_0\in\calH$ and $u\in L^1_{\rm loc}(0,\infty;\C^m)$, it is in general not guaranteed that a classical solution of \eqref{e:dyn} exists. In fact, for general semigroups, the existence of a classical solution require two ingredients: First, clearly $x_0\in\dom A$ and second continuous right-hand sides are required, e.g., $Bu\in H^1(0,T;\calH)$ or $Bu\in C(0,T;\dom A)$ to ensure the existence of a classical solution, cf.\ \cite[Cor.\ VI.7.6, Cor.\ VI.7.8]{engelnagel}. This framework is quite restrictive in terms of optimal control, where one might aim to choose, e.g., piecewise constant control inputs. Therefore, we consider mild solutions here. The (mild) solution of \eqref{e:dyn} is defined by
\begin{align}
\label{e:varofconst}
x(t;x_0,u) := \mathcal{T}(t)x_0 + \int_0^t\mathcal{T}(t-s)Bu(s)\,ds.
\end{align}
Note that the mild solution is always a continuous $\calH$-valued function.

We shall frequently make use of the following auxiliary lemma, which easily follows from the equivalence $x_n\to x$ in $L^2(0,T;X)$ iff $\|x_n-x\|\to 0$ in $L^2(0,T;\R)$.

\begin{lem}\label{l:subsequence}
Let $X$ be a Hilbert space and $(x_n)\subset L^2(0,T;X)$ such that $x_n\to x$ in $L^2(0,T;X)$. Then there exists a subsequence $(x_n')$ of $(x_n)$ such that $x_n'(t)\to x(t)$ in $X$ for a.e.\ $t\in [0,T]$.
\end{lem}

Let $T>0$. In what follows, we shall eventually (and particularly in the OCP \eqref{e:ocp}) restrict the controls $u$ on time horizons $[0,T]$ to the function class
$$
\calU_T := \big\{u\in L^2(0,T;\C^m) : u(t)\in\UU\text{ a.e.}\big\},
$$
where $\UU\subset\C^m$ is a compact and convex set having the origin in its interior $\inte\UU$. Note that $\calU_T$ is bounded, closed (Lemma \ref{l:subsequence}), and convex in $L^2(0,T;\C^m)$.

We say that $x_1\in\calH$ is reachable from $x_0\in\calH $ at time $T$, if there exists $u\in\calU_T$ such that $x(T;x_0,u) = x_1$. By $\calR_T(x_0)$ we denote the set of all states which are reachable from $x_0$ at time $T$. Moreover, we set $\calR_0(x_0) := \{x_0\}$ and
$$
\calR(x_0) := \bigcup_{T\ge 0}\calR_T(x_0).
$$

\subsection{Higher regularity of mild solutions and dissipativity}
The energy Hamiltonian of the system \eqref{e:dyn}--\eqref{e:out} is given by
$$
H(x) := \tfrac 12\|x\|^2.
$$
One of the central identities for finite-dimensional port-Hamiltonian systems is a {\em dissipation equality}. If $\calH$ is finite-dimensional, the dissipation equality reads
$$
\int_0^T\!\!\Re\<u,y\>\,dt = H(x(T))-H(x(0)) + \int_0^T\!\!\|R^\frac 12x\|^2\,dt
$$
and holds for all solutions of the system \eqref{e:dyn}--\eqref{e:out}. It is clear that, if $R$ is unbounded, this energy balance only makes sense if $x(t)\in\dom R^{\frac12}$ for a.e.\ $t$. Whereas this is clearly the case for classical solutions $x\in C(0,T;\dom A)\cap C^1(0,T;\calH)$, the next proposition shows that this regularity property also holds for mild solutions.

\begin{prop}\label{p:higherreg}
For each $x_0\in\calH$ and each $u\in L^1(0,T;\C^m)$ the mild solution $x = x(\,\cdot\,;x_0,u)$ satisfies $x(t)\in\dom R^{\frac 12}$ for a.e.\ $t\in [0,T]$. We have $R^{\frac 12}x\in L^2(0,T;\calH)$ and
\begin{equation}\label{e:diss}
\int_0^T\!\!\Re\<Bu(t),x(t)\>\,dt = (H\circ x)\Big|_0^T + \int_0^T\!\!\|R^\frac 12x(t)\|^2\,dt.
\end{equation}
In particular, we have $\calT(t)x_0\in\dom R^{\frac 12}$ for a.e.\ $t>0$ and $R^{\frac 12}\calT(\cdot)x_0\in L^2(0,T;\calH)$.
\end{prop}
\begin{proof}
First of all, assume that $x_0\in\dom A$ and $u\in C^1(0,T;\C^m)$. The mild solution $x := x(\,\cdot\,;x_0,u)$ is then a classical solution, i.e., $x\in C^1(0,T;\calH)$ and $x(t)\in\dom A$ for all $t\in [0,T]$. For $t\in [0,T]$ we obtain
\begin{align*}
\tfrac d{dt}H(x(t))
&= \Re\<\dot x(t),x(t)\> = \Re\<Ax(t) + Bu(t),x(t)\>\\
&= \Re\<Bu(t),x(t)\> -\<Rx(t),x(t)\>.
\end{align*}
Integrating this shows that $(u,x)$ satisfies \eqref{e:diss}.

Now, let $x_0\in\calH$ and $u\in L^1(0,T;\C^m)$ be arbitrary. Then we find sequences $((x_0)_n)\subset\dom A$ and $(u_n)\subset C^1(0,T;\C^m)$ such that $(x_0)_n\to x_0$ in $\calH$ and $u_n\to u$ in $L^1(0,T;\C^m)$. Set $x_n := x(\,\cdot\,;(x_0)_n,u_n)$. For $t\in [0,T]$ we have
$$
\|x_n(t) - x(t)\|\le\|(x_0)_n - x_0\| + \|B\|\|u_n-u\|_{L^1},
$$
which implies that $x_n\to x$ uniformly on $[0,T]$.

Set $u_{nm} := u_n-u_m$ and $x_{nm} := x_n-x_m$. Then $x_{nm} = x(\,\cdot\,;(x_0)_n-(x_0)_m,u_{nm})$. Hence, $(u_{nm},x_{nm})$ satisfies \eqref{e:diss} and we obtain $\int_0^T\|R^{\frac 12}x_{nm}(t)\|^2\,dt\to 0$ as $n,m\to\infty$. That is, $(R^{\frac 12}x_n)$ is a Cauchy sequence in $L^2(0,T;\calH)$ and thus converges to some $z\in L^2(0,T;\calH)$. In particular, there exists a subsequence $(x_n')$ of $(x_n)$ such that $R^{\frac 12}x_n'(t)\to z(t)$ for a.e.\ $t\in [0,T]$, see Lemma \ref{l:subsequence}. From the closedness of $R^{\frac 12}$ we infer that $x(t)\in\dom R^{\frac 12}$ for these $t$ and $R^{\frac 12}x(t) = z(t)$ so that $R^{\frac 12}x = z\in L^2(0,T;\calH)$. It is now clear that \eqref{e:diss} is satisfied.
\end{proof}

In view of Proposition~\ref{p:higherreg}, we can provide the following system-theoretic property of \eqref{e:dyn}.

\begin{cor}
Considering the supply rate $\Re\langle Bu,x\rangle$, the dynamics \eqref{e:dyn} are dissipative in the following sense: for any $x_0\in\calH$ and $u\in L^1(0,T;\C^m)$ the solution of \eqref{e:dyn} with $x(0)=x_0$ satisfies \eqref{e:diss}.
\end{cor}

\subsection{Existence of optimal solutions}
In this part we shall provide a brief proof of the following existence theorem.

\begin{thm}
If  $x_T \in \calR_T(x_0)$, then there is an optimal solution $u^\star \in \calU_T$ of \eqref{e:ocp}.
\end{thm} 
\begin{proof}
The set of admissible controls
$$
\calU_\text{ad}:=\{u\in \calU_T\,|\,x(T;x_0,u)=x_T\}
$$
is bounded, closed and convex due to the corresponding properties of $\calU_T$ and linearity and boundedness of $\vphi : L^2(0,T;\C^m)\to\calH$, $\vphi(u):=\int_0^T\calT(T-t)Bu(t)\,dt$. Utilizing the dissipativity equality \eqref{e:diss}, the cost functional $\int_0^T \Re\langle Bu(t),x(t)\rangle\,dt $ can equivalently be replaced by
\begin{align}\label{e:cost_R}
J(u):=\int_0^T\|R^{\frac 12}x(t)\|^2\,dt = \|Cu\|_{L^2(0,T;\calH)}
\end{align}
with some bounded affine-linear operator
$$
C : L^2(0,T;\C^m) \to L^2(0,T;\calH),
$$
i.e., the cost functional $J$ is convex and continuous and hence weakly lower semi-continuous.

Let $(u_n)\subset\calU_\text{ad}$ such that $J(u_n)\to a := \inf\{J(u)\,|\,u\in \calU_\text{ad}\}$. Boundedness of $\calU_\text{ad}$ implies that there is some $u^\star\in L^2(0,T;\C^m)$ such that $u_n\wto u^\star$. As $\calU_\text{ad}$ is closed and convex, it is weakly sequentially closed, so that $u^\star\in \calU_\text{ad}$. Now, the lower semi-continuity of $J$ implies
\begin{align*}
a \leq J(u^\star) = J\left(\lim_{n\to\infty} u_n\right) \leq \lim_{n\to \infty} J(u_n) = a,
\end{align*}
and hence $J(u^\star)=a$. As $u^\star\in\calU_\text{ad}$, it follows that $u^\star$ is an optimal solution of \eqref{e:ocp}.
\end{proof}

\section{Turnpike analysis of optimal solutions} \label{s:TP}
In \cite{Schaller2021a} we considered the port-Hamilto\-ni\-an OCP \eqref{e:ocp} in the finite-dimensional setting, i.e., $\dim\calH < \infty$, and proved that optimal solutions exhibit a {\em turnpike} towards the conservative subspace $\ker R$, i.e., they stay close to this subspace for the majority of the time. Since the situation in the infinite-dimensional case is obviously more involved, let us first evaluate some numerical experiments.

\subsection{Numerical Examples}
We consider two examples: (a) the diffusion equation and (b) the Timoshenko beam. In example (a) we have $J=0$, whereas $R$ is bounded in example (b). Hence, both cases fit into our setting.

{\bf Diffusion equation}. We consider a diffusion equation with Neumann boundary conditions on a bounded domain $\Omega\subset \R^n, n=2,3$, with locally Lipschitz boundary. That is, we set $\calH = L^2(\Omega,\R) = L^2(\Omega)$, $J=0$, and considering a constant diffusivity $d>0$, we define the diffusion operator
$$
R = -d\Delta, \quad \dom R = \left\{x\in H^2(\Omega)\,:\,\frac{\partial x}{\partial \nu}=0\right\},
$$
where $\frac{\partial x}{\partial \nu}=\langle \nabla x,\nu\rangle$ denotes the outward unit normal derivative with outward unit normal $\nu$ of $\Omega$. It is well known that the operator $R$ is self-adjoint and non-negative. This system is also called Dissipative Port Hamiltonian as it may be decomposed into the system
\begin{equation}
\begin{pmatrix}
\frac{\partial x}{\partial t} \\
F 
\end{pmatrix} =
\begin{pmatrix}
0 & - div\\
- grad & 0
\end{pmatrix}
\begin{pmatrix}
\frac{\partial G}{\partial x} \\
\Phi  
\end{pmatrix}
\end{equation} where $G(x)$ is Gibbs' free energy,  $\Phi$ is the flux of species, $F$ is the driving force of the diffusion and completed with the Gibbs' diffusion relation $\Phi=d\:F$. The matrix differential operator is well-know to be a Hamiltonian operator which may be extended to a Stokes-Dirac structure encompassing boundary port variables being the the mass flux and the chemical potential at the boundary \cite{Baaiu09b}. Note that the Neumann boundary condition corresponds to a zero driving force at the boundary, corresponding to the thermodynamical equilibrium.

As the control operator we let
$$
(Bu)(t,\omega) := \sum_{i=1}^{m} \psi_i(\omega)u_i(t),
$$
where $(\psi_i)_{1\leq i\leq m} \subset  L^2(\Omega)$ is a given family of shape functions modeling the mass influx distributed 
over the domain (for instance, due to the varying porosity of a membrane). The conjugated output is then given by
$$
y = B^*x = \begin{pmatrix}
\int_\Omega \psi_1(\omega)x(\omega)\,d\omega\\
\dots\\
\int_\Omega \psi_m(\omega)x(\omega)\,d\omega\\
 \end{pmatrix}
$$ and is the mean chemical potential.

Motivated by the results and observations in \cite{Schaller2021a}, we expected a turnpike behavior of optimal solutions of \eqref{e:ocp} towards the subspace $\ker R^{\frac12}$, which obviously consists of the constant functions. In other words, we expected $R^{\frac 12}x(t)$ to perform a turnpike towards zero. And indeed, our numerical experiments confirm this expectation.

As an example, we consider the diffusion of a species on a one dimensional domain described by $\Omega = [0,1]$, $m=1$, and 
$$
\psi_1(\omega)= \begin{cases}
1 \quad &\omega\in [1/2-\delta, 1/2+\delta]\\
0 \quad &\text{else},
\end{cases}$$ 
for some $\delta >0$, i.e., we have one actuator with diameter $2\delta$ in the center of the one-dimensional domain. Further, we set the diffusivity to $d = 0.1$. We discretize the rod with 21 grid points in space and standard finite differences. Concerning time discretization of $[0,T]$, where $T = 5,10,20,40$, we use $N=251,501,1001,2001$ (resp.) shooting intervals and a Runge Kutta scheme of order four on all shooting intervals. As initial condition and terminal condition, we set $x_0(\omega) = \sin(\pi \omega)$ and $x_T(\omega)\equiv 2$.

In Figure~\ref{fig:statecontrol} we depict the optimal solutions of \eqref{e:ocp} for different optimization horizons. We observe that the control $u$ exhibits a classical turnpike property towards $u=0$, i.e., it stays close to the origin for the majority of the time, while the state $x$ obviously does not show such a behavior.
\begin{figure}[!ht]
\centering	
\includegraphics[scale = 0.5]{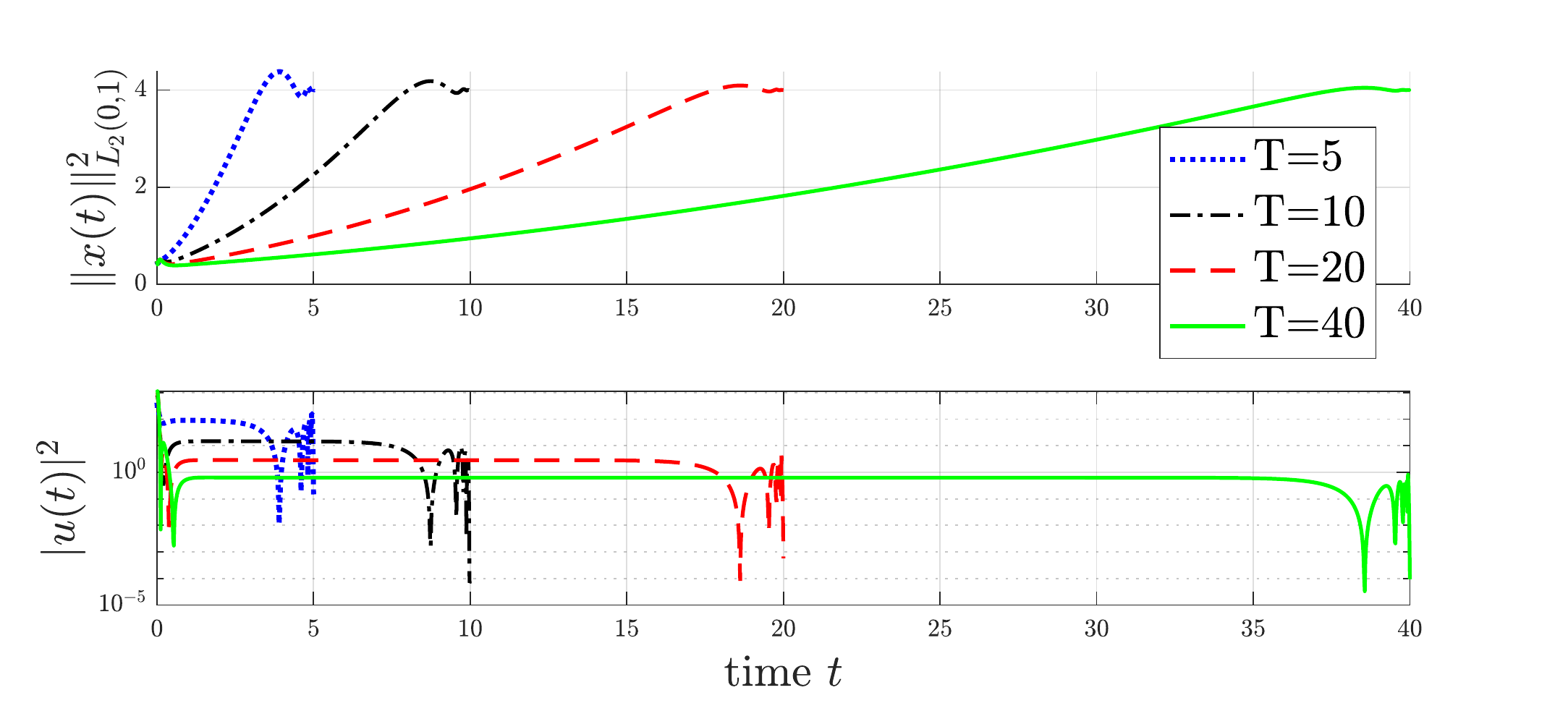}
\caption{Energy of optimal state and control over time for different optimization horizons for the diffusion equation.}\label{fig:statecontrol}
\end{figure}
However, in Figure~\ref{fig:turpike} we observe the expected turnpike behavior towards $\ker R^{\frac12}$. We also notice a smaller distance to the turnpike when increasing the time horizons. Further, $R^{\frac 12}x(t) = 0$ is only reached exactly at the terminal time due to $x_T\in \ker R^{\frac12}$.
\begin{figure}[!ht]
\centering	
\includegraphics[scale = 0.5]{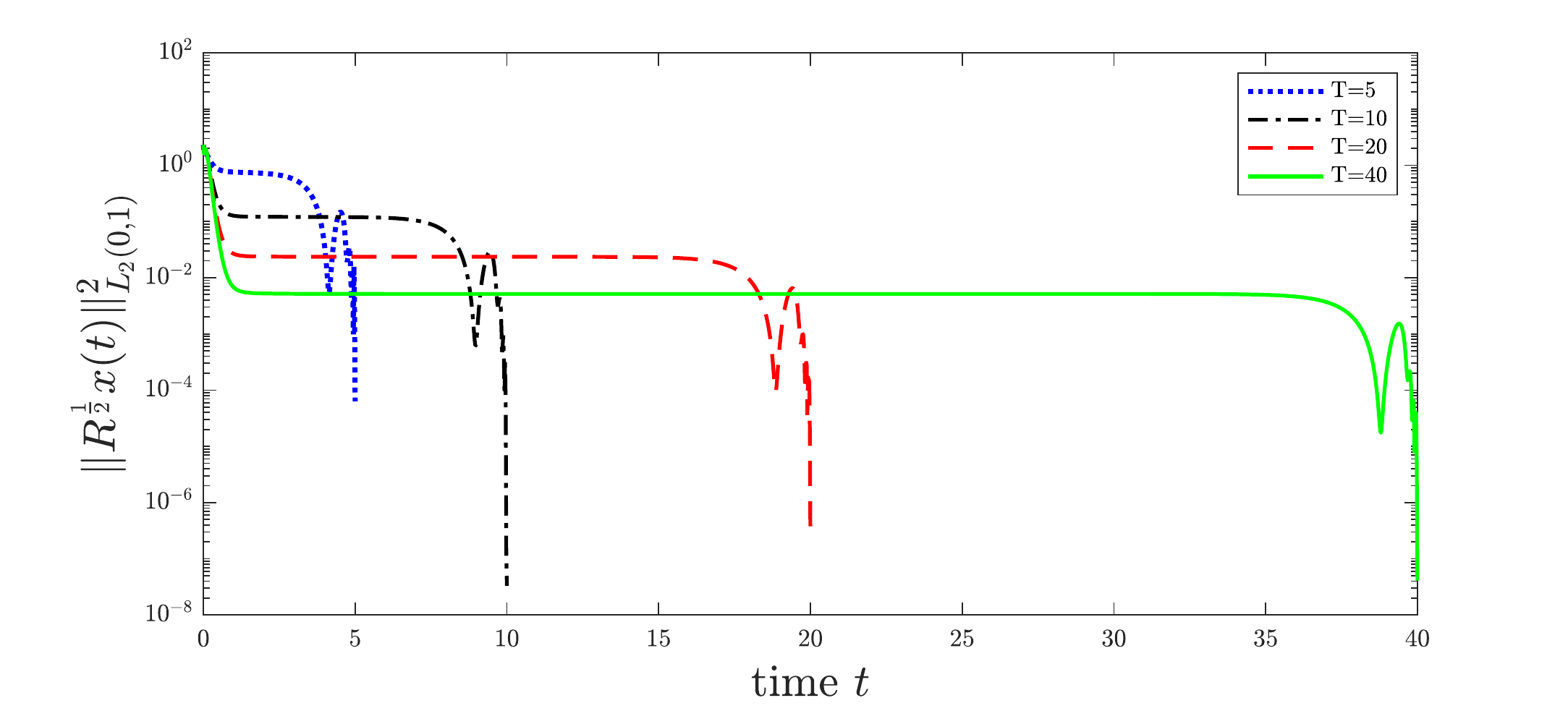}
	\caption{The function $\|R^{\frac 12}x(t)\|_{L^2(0,1)}^2$ for different optimization horizons for the diffusion equation.}\label{fig:turpike}
\end{figure}
In Figure~\ref{fig:snapshots} the distribution of species at time $t=T/2$ is depicted. We see that with an increasing value of the horizon $T$, the corresponding snapshot becomes more and more constant in space, i.e., closer to the conservative subspace $\ker R$.
\begin{figure}[!ht]
	\centering
	\includegraphics[scale = 0.5]{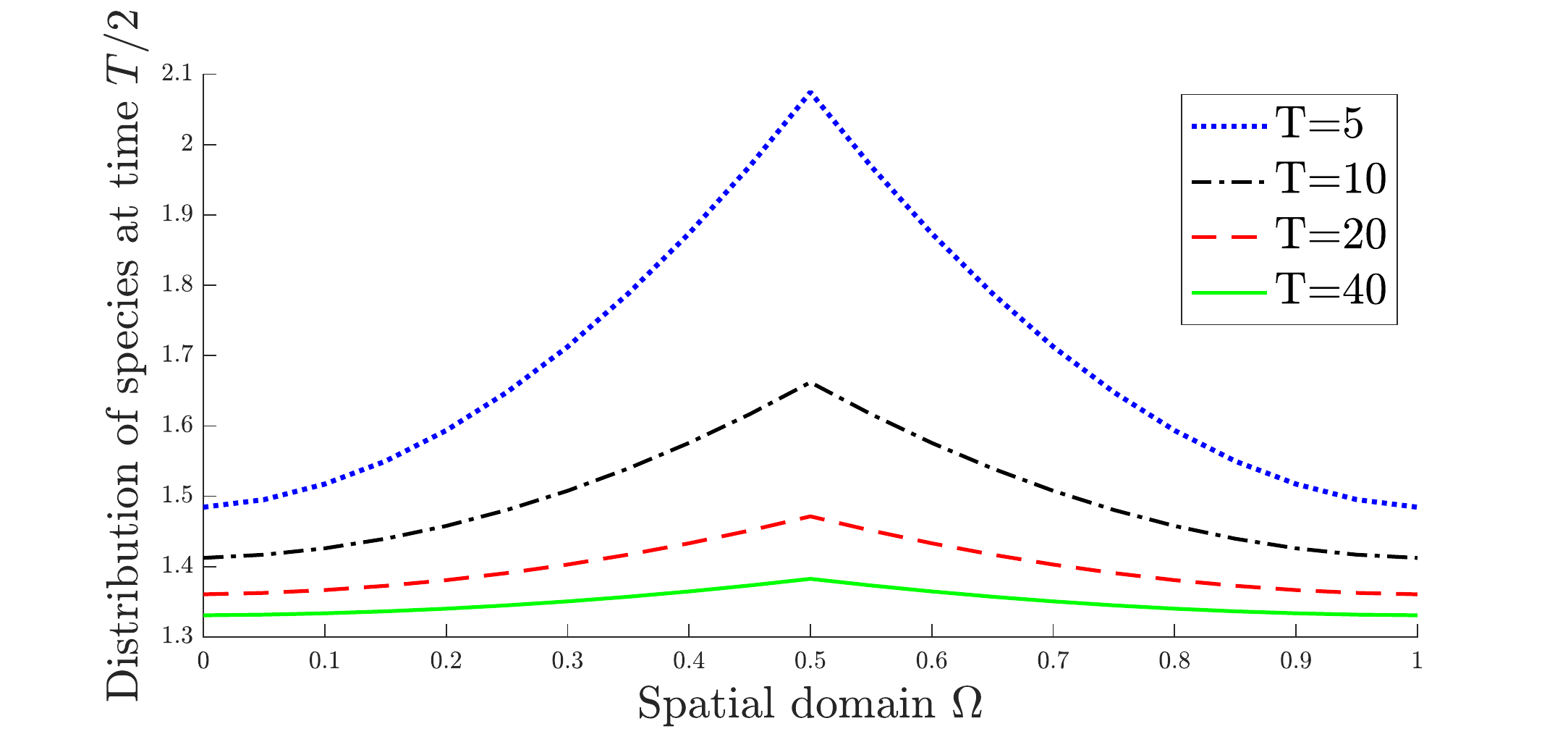}\\
   \includegraphics[scale = 0.5]{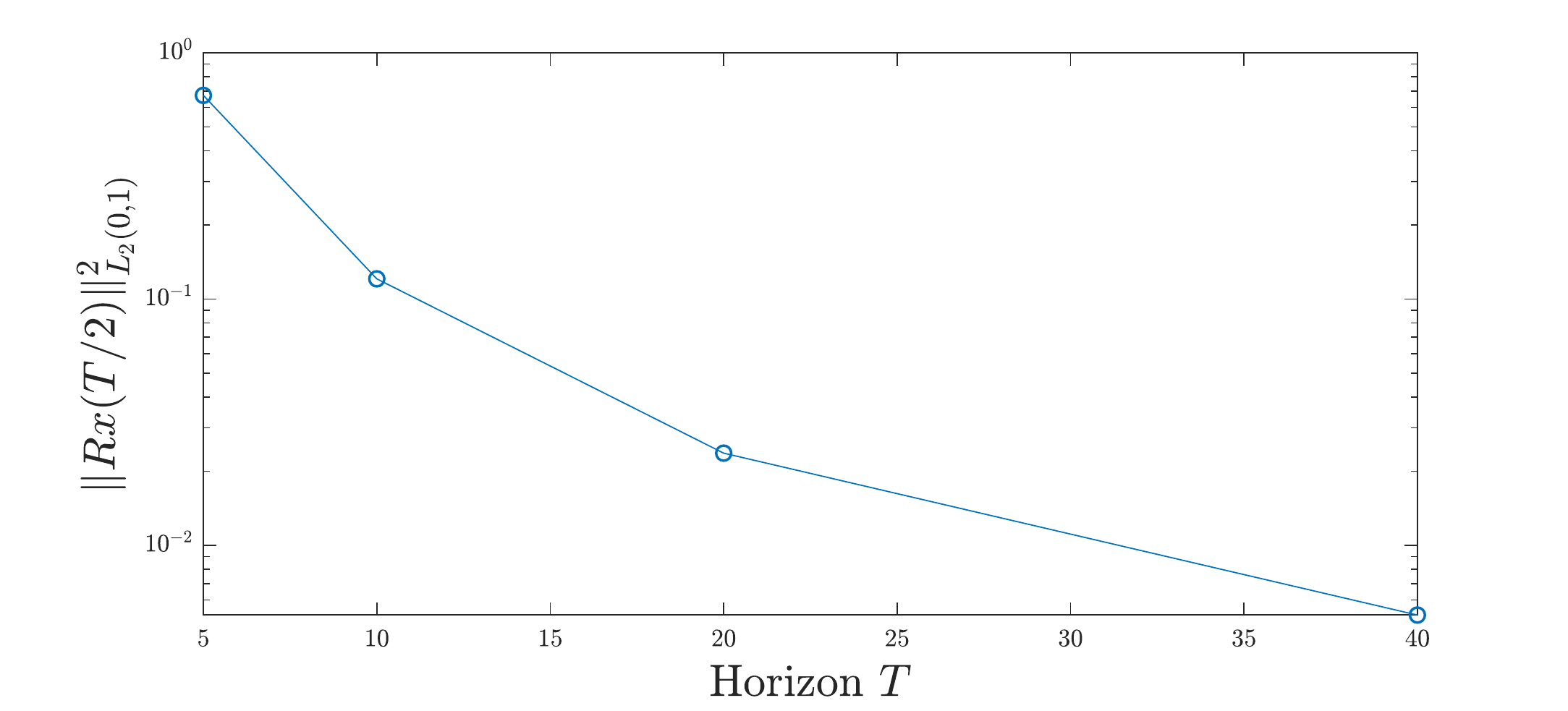}
	\caption{Top: Snapshots of the species distribution at time $t=T/2$ for different optimization horizons for the diffusion equation. Bottom: Distance to the subspace of thermodynamical equilibria at the midpoint of the time interval for different optimization horizons.}\label{fig:snapshots}
\end{figure}

{\bf Timoshenko beam}. We consider the example from \cite{Wu2021}, cf.\ also \cite[Example 7.1.4]{Jacob2012} of a Timoshenko beam in the state space $\calH = L^2(0,1;\mathbb{R}^4)$. The coordinates of $x\in L^2(0,1;\mathbb{R}^4)$ can be interpreted as follows: $x_1$ is the shear displacement, $x_2$ the transverse momentum, $x_3$ the angular displacement, and $x_4$ the angular momentum. As to the dynamics, we consider the unbounded skew-adjoint operator 
\begin{align*}
J := \begin{pmatrix}
0&\tfrac{\partial}{\partial z} &0 &-1\\
\tfrac{\partial}{\partial z} & 0&0&0\\
0&0&0&\tfrac{\partial}{\partial z}\\
1&0&\tfrac{\partial}{\partial z}&0
\end{pmatrix},
\end{align*}
defined on the domain 
$$
\dom J = \left\{x\in H^1(0,1;\R^4) \,:\,\begin{matrix}
x_2(0)=0\\
x_4(0)=0\\
x_1(1)=0\\
x_3(1)=0
\end{matrix} \right\}.
$$
The boundary conditions are modeling the fact that the beam is clamped in the left-hand side, i.e., shear and angular displacement vanish on the unclamped side and transverse and angular momentum vanish on the clamped side.

Further, we assume that we have internal damping. For this, consider the self-adjoint and bounded operator
$$
R = \operatorname{diag}(0,R_1,0,R_2),
$$
with $R_1,R_2>0$, i.e., the transverse momentum is damped by $R_1$ and the angular momentum is damped by $R_2$. For simplicity we set the matrix operator $Q$ from \cite{Wu2021} to $Q = I$.

The control of the beam is put into effect by two torques, one on the left half and one on the right half of the beam, which we can actuate by a scalar control input each. I.e.,
$$
(Bu)(z):= \begin{pmatrix}
0\\
0\\
0\\
\chi_{[0,\nu]}(z)u_1 + \chi_{[1-\nu,1]}(z)u_2
\end{pmatrix},
$$
where $\nu\in (0,1/2)$, $u = (u_1,u_2)^\top\in\R^2$ and $\chi_S(z)$ de\-notes the characteristic function of a set $S\subset [0,1]$. The corresponding conjugated output for $x\in L^2(0,1;\R^4)$ is given by 
$$
y_1 = \int_0^\nu x_4(z)\,dz\quad  \text{ and }\quad y_2 = \int_{1-\nu}^1 x_4(z)\,dz,
$$
respectively.

Also in this case, we expected and observed a turnpike behavior of optimal solutions towards the subspace
\begin{align*}
\ker R = \left\{x\in L^2(0,1;\R^4) : x_2(z)=x_4(z)=0 \text{ a.e.}\right\}.
\end{align*}
For numerical solution, we consider a finite difference approximation and $\nu = 0.5$. In Figure~\ref{fig:timoshenko} we depict the optimal solutions with $x_0=x_T \equiv 1$. We observe that the states $x_2$ and $x_4$ clearly show a turnpike behavior towards zero, while $x_3$ does not.
\begin{figure}[ht]
	\centering	
	\includegraphics[scale = 0.5]{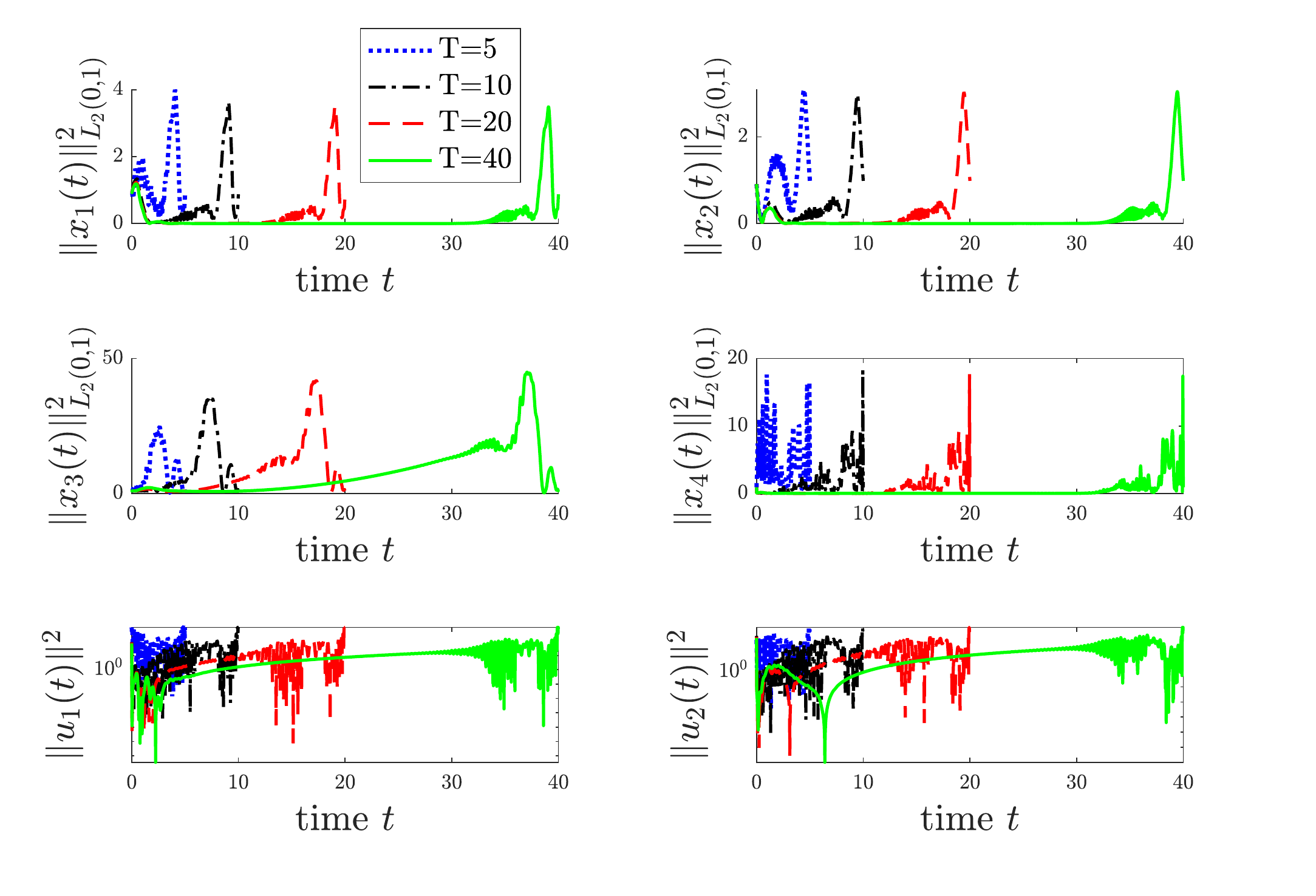}
	\caption{Energy of optimal states and control over time for different optimization horizons for the Timoshenko beam with constant one function as initial and terminal condition.}\label{fig:timoshenko}
\end{figure}
 In Figure~\ref{fig:timoshenko2} we depict the optimal solutions with $x_0(\omega) =x_T(\omega) = (\omega,1-\omega,\omega,1-\omega)^\top$ and observe the same behavior, i.e., $x_2$ and $x_4$ are close to zero in the middle part of the horizon, whereas $x_3$ is not.
\begin{figure}[ht]
	\centering	
	\includegraphics[scale = 0.44]{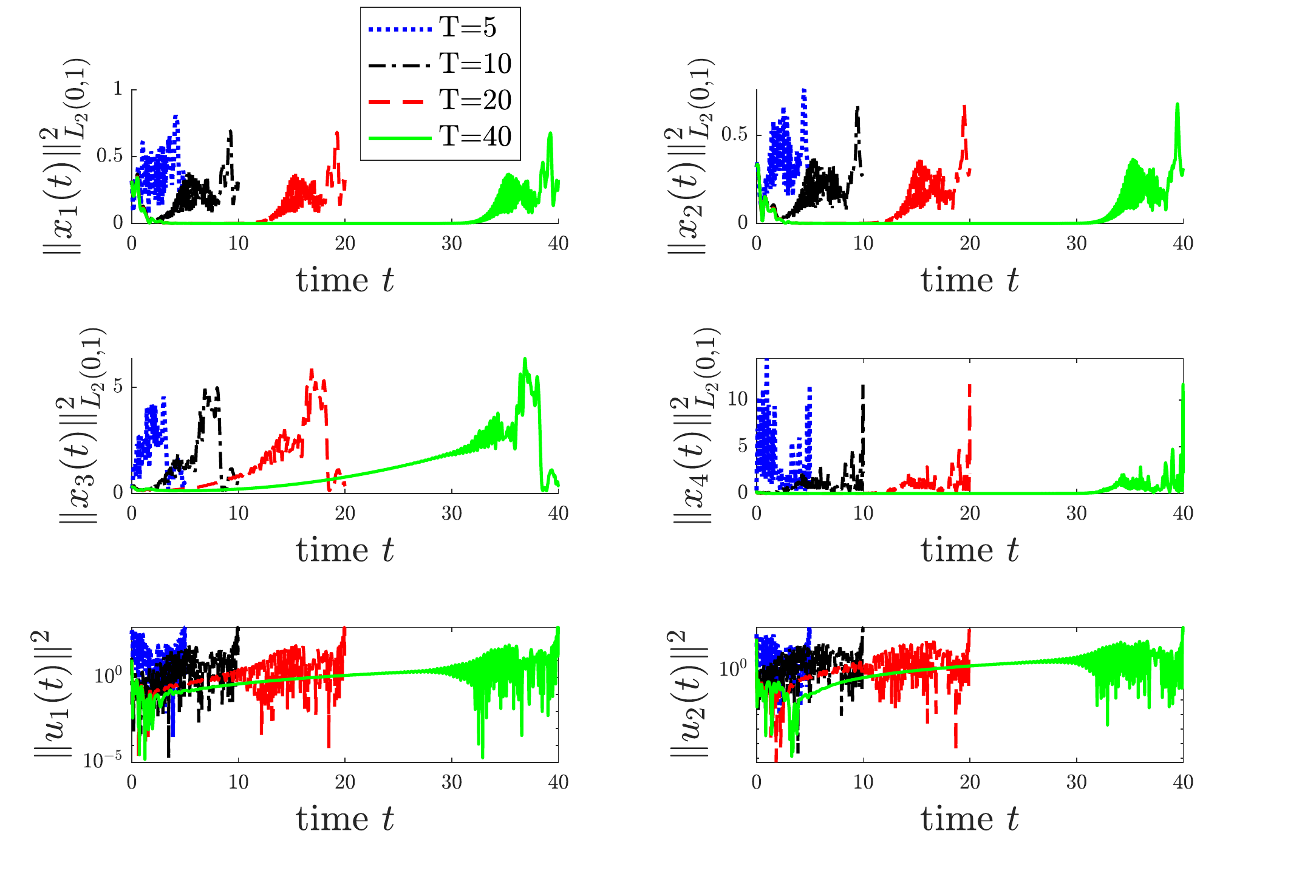}
	\caption{Energy of optimal states and control over time for different optimization horizons for the Timoshenko beam with linear initial and terminal condition in space.}\label{fig:timoshenko2}
\end{figure}

\subsection{A proof of subspace turnpike behavior}\label{sec:OCP}
In this part we formalize and give sufficient conditions for the observed subspace turnpike property for infinite-dimensional optimal control problems.

\begin{defn}[State Subspace Turnpike Property]\label{d:turnpike}
Let $A:\dom A\subset \calH\to \calH$ be a generator of a $C_0$-semigroup in $\calH$ and let $B\in L(\C^m,\calH)$. We say that a general OCP with linear dynamics of the form
\begin{align}
\begin{split}\label{e:lin_OCP}
\min_{u\in \calU_T}\,&\int_0^T \ell(x(t),u(t))\,dt\\
&\dot x = Ax + Bu\\
&x(0)=x_0,\quad x(T)=x_T
\end{split}
\end{align}
with $x_0,x_T\in\calH$ and a $C^1$-function $\ell : \calH\times\C^{m}\to\R$ has the {\em state integral turnpike property} on a set $S_{\rm tp}$ with respect to a subspace $\mathcal{V}\subset\calH$, if there exist continuous functions $F,T : S_{\rm tp}\to [0,\infty)$ such that for all $x_0\in S_{\rm tp}$ each optimal pair $(x^*,u^*)$ of the OCP \eqref{e:lin_OCP} with initial datum $x^*(0)=x_0$ and $T > T(x_0)$ satisfies
\begin{align}\label{e:integral_tp}
\int_0^T\dist^2(x^*(t),\calV)\,dt\,\le\,F(x_0).
\end{align}
\end{defn}

The following theorem provides the main result of this work and shows that a turnpike property is present under reachability assumptions. Recall that the {\em discrete spectrum} $\sd(R)$ of $R$ is the collection of all isolated eigenvalues of $R$.

\begin{thm}\label{t:turnpike}
Assume that $0\in\sd(R)\cup\rho(R)$. Let $T_0 > 0$ and set $\mathbb{X} := \{x_0\in\calH : 0\in\calR_{T_0}(x_0)\}$. Furthermore, assume that $x_T\in\calR(0)$. Then the OCP \eqref{e:ocp} has the state integral turnpike property on $\mathbb{X}$ with respect to $\calV = \ker R$.
\end{thm}
\begin{proof}
By assumption, there exist $T_1\ge 0$ and $u_j\in\calU_{T_j}$, $j=0,1$, such that $x(T_0;x_0,u_0)=0$ and $x(T_1;0,u_1)=x_T$. Set $T(x_0) := T_0+T_1$ for $x_0\in\mathbb{X}$ (i.e., $T(\cdot)$ is constant) and
\begin{align*}
u(t) :=\begin{cases}
u_0(t),   &t\in [0,T_0]\\
0  &t\in [T_0,T-T_1]\\
u_1(t-(T-T_1)),  &t\in [T-T_1,T].
\end{cases}
\end{align*}
By $x$ denote the corresponding state response, i.e.,
\begin{align*}
x(t)=
\begin{cases}
\mathcal{T}(t)x_0 + \int_0^t\mathcal{T}(t\!-\!s)Bu_0(s)\,ds &t\in [0,T_0]\\
0 &t\in [T_0,T\!-\!T_1]\\
\int_{T-T_1}^t\!\!\!\mathcal{T}(t\!-\!s)Bu_1(s\!-\!(T\!-\!T_1))\,ds &t\in [T\!-\!T_1,T].
\end{cases}
\end{align*}
Let $(x^*,u^*)\in C(0,T;\calH)\times \calU_T$ be an optimal state-control pair for \eqref{e:ocp} and $y^*\in C(0,T;\C^m)$ the corresponding output. Then for the equivalent cost functional $J$ (see \eqref{e:cost_R}) we have $J(u^*)\le J(u)$, i.e.,
\begin{align*}
\int_0^T\left\|R^{\frac12}x^\star(t)\right\|^2\,dt \leq  \int_0^T\left\|R^{\frac12}x(t)\right\|^2\,dt = \int_0^{T_0}\left\|R^{\frac12}x(t)\right\|^2\,dt + \int_{T-T_1}^{T}\left\|R^{\frac12}x(t)\right\|^2\,dt,
\end{align*}
since $x(t)=0$ for all $t\in [T_0,T-T_1]$. As $(\mathcal{T}(t))_{t\geq 0}$ generates a contraction semigroup, we estimate
\begin{align*}
\int_0^{T_0}\left\|R^{\frac12}x(t)\right\|^2\,dt
&= \int_0^{T_0}\Re\<Bu_0(t),x(t)\rangle\,dt - (H\circ x)\Big|_0^{T_0}\\
&\le H(x_0) + \|B\|\|u_0\|_{L_1(0,T_0;\C^m)}\|x\|_{L^\infty(0,T_0;\calH)}\\
&\le H(x_0) + \|B\|T_0u_{\max}\cdot\big(\|x_0\| + \|B\|T_0u_{\max}\big),
\end{align*}
where $u_{\max}:=\max_{v\in\UU}\|v\|$. Setting $\wt x(t):= x(t;0,u_1)$, we have $\int_{T-T_1}^T\|R^{\frac 12}x(t)\|^2\,dt = \int_0^{T_1}\|R^{\frac 12}\wt x(t)\|^2\,dt$ and thus
\begin{align*}
\int_{T-T_1}^T\|R^{\frac 12}x(t)\|^2\,dt\le\|B\|^2T_1^2u_{\max}^2.
\end{align*}
Hence, we have $\int_0^T\|R^{\frac12}x^\star(t)\|^2\,dt\le G(x_0)$ with a continuous function $G : \mathbb X\to [0,\infty)$, which is independent of~$T$.

Let $E$ denote the spectral measure of the self-adjoint operator $R$ and let $\sigma_+ := \min\sigma(R)\backslash\{0\}$, which is a positive value by assumption. Then for $z\in\dom R^{\frac 12}$ we have
\begin{align*}
\|R^{\frac 12}z\|^2
&= \int_{\sigma_+}^\infty\la\,d\|E_\la z\|^2\ge\sigma_+\cdot\|E([\sigma_+,\infty))z\|^2 = \sigma_+\cdot\dist^2(z,\ker R).
\end{align*}
Therefore, we obtain
$$
\int_0^T\dist^2(x^\star(t),\ker R)\,dt\,\le\,\frac{G(x_0)}{\sigma_+} =: F(x_0),
$$
which finishes the proof of the theorem.
\end{proof}

\begin{rem}
Let us briefly comment on the assumptions of the previous theorem.

(a) The imposed reachability assumptions can be relaxed in various ways: First, controllability could be replaced with stabilizability, if the stabilizing feedback satisfies the control constraints. Second, the terminal constraint $x(T)=x_T$ could be replaced by an inclusion $x(T)\subset \Psi$, where $\Psi\subset \calH$. Further, the role of the zero in $0\in R_{T_0}(x_0)$ and $x_T\in \calR(0)$ could be replaced by any controlled equilibrium in $\ker R$.

(b) The assumption $0\in \sigma_\text{d}(R) \cup \rho(R)$ is fulfilled, if $R$ has compact resolvent.

(c) If $\calH$ is finite-dimensional and the Kalman matrix of $(A,B)$ as full rank, one can show that the assumption $0\in \calR_{T_0}(x_0)$ is satisfied for some $T_0$ (and then for all $T>T_0$, as well) due to the spectral properties of $J-R$ and the assumption $0\in\inte{\UU}$, We are not aware of a corresponding result in the infinite-dimensional case.
\end{rem}

\section{Conclusions} \label{sec:con}

This paper has investigated optimal control problem for infinite-dimensional linear port-Hamiltonian systems under the intrinsic objective of minimizing the supplied energy. 
We have shown that similarly to the finite-dimensional case, this singular optimal control problem exhibits a turnpike phenomenon with respect to the subspace induced by the dissipation operator. 
We have motivated our findings via simulations of the Timoshenko-beam and the diffusion equation. 

\bibliographystyle{abbrv}
\bibliography{pHInfDim_Arxiv.bib}
\end{document}